\documentclass[]{article}
\usepackage{amssymb,amsmath,amsfonts,soul,amsthm,enumerate}

\newtheorem{theorem}{Theorem}[section]

\newtheorem{definition}{Definition}

\newtheorem{lemma}{Lemma}[section]

\usepackage{float}
\usepackage{calc,tikz}
\usetikzlibrary{matrix}
\usetikzlibrary{arrows,decorations.markings,positioning,shapes.geometric}
\tikzstyle{vertex}=[circle, draw, inner sep=1pt, minimum size=4pt]
\newcommand{\vertex}{\node[vertex]}

\tikzstyle{ann} = [fill=white,font=\footnotesize,inner sep=1pt]
\tikzstyle{arrow} = [thick,<-->,>=stealth]
\title{Maximum Number of Non-Intersecting Diagonals in Square Arrays}
\author{Marbarisha M. Kharkongor and Joseph Varghese Kureethara\thanks{Christ University, Bangalore, India}}
\date{}
\begin{document}

\maketitle

\begin{abstract}
In this paper, we derive a formula to express the maximum number of non-intersecting diagonals of arbitrary length that can be drawn in $n \times n$ square arrays, where $n$ is a multiple of $l+1$.
\end{abstract}

\section{\large Introduction}

Spatial filling problems have been extensively explored and examined over the years and the results obtained along the way are important as they are applied to various branches of science. In this paper, we shall be study one aspect of such a filling problem. The problem of finding the maximum number of non-intersecting diagonals of length 1 that can be accommodated in an array \cite{a} is the motivation of this paper.

The problem of finding maximum number of non-intersecting diagonals was reminiscent of classical chessboard problems that involve placing the maximum number of certain types of chess pieces (e.g., queens, bishops, rooks or pawns) on the board so that no piece attacks another \cite{d}. 

In this paper, we will generalize the results to diagonals of arbitrary length subject to certain constraints.

This work is an improvement of a study already done on the maximum number of non-intersecting diagonals of length 1 (length of a diagonal is elaborated later in this paper) in an array \cite{a}. In the earlier study by Boyland \textit{et al}.\cite{a}, extensive work has been done in obtaining results for diagonals of length 1 in square as well as rectangular arrays.

In the context of the earlier work \cite{a}, this paper explores the idea of diagonals of any finite length, where a general formula for the maximum number of non-intersecting diagonals of any finite length $l$ is obtained as the final result. However, generalising the length of the diagonals brings about certain complexities which require certain conditions to be imposed so as to get meaningful results. In this paper, only positive diagonals are considered in square arrays of size $n \times n$ where $n$ is a multiple of $l+1$.

There is also a scope for further study on this topic, where the constraints may be removed. The problem can also be extended to diagonals of finite length in 3-dimensional space.

An array is a systematic arrangement of similar objects in rows and columns. In this paper, we deal with two dimensional square arrays. An $n \times n$ square array will be denoted as an $n$-array. A diagonal of a unit square in an array is a straight line joining the opposite corners (or opposite lattice points) of the unit square.

In this paper, we will be considering only positive diagonals, $i.e.,$ diagonals having positive slope.

A diagonal of a unit square in an array can be thought of as having length 1, and is denoted as a 1-diagonal.
Two 1-diagonals in an array are said to intersect if they share a common lattice point or if they belong to the same unit square.
\begin{definition}
	Consider $l$ consecutive unit squares in an array, such that any two adjacent unit squares will have exactly one lattice point in common. A diagonal of length $l$ is a diagonal which is contained in all of these $l$ unit squares in such a way that the diagonal passes through the common lattice points. Such a diagonal is denoted as a $l$-diagonal.
\end{definition}
Any two $l$-diagonals in an array are said to intersect if any of the 1-diagonals that form them intersect each other.

\begin{definition}\label{def:Larran}
	The L-arrangement is a way of arranging diagonals in an $n$-array. This method of arrangement is as follows:
	\begin{enumerate}
		\item $l$-diagonals are drawn starting from the leftmost column of the array in such a way that each diagonal occupies the first $l$ leftmost columns. Then $l$-diagonals are drawn starting from the bottom row in such a way that each diagonal occupies the last $l$ bottom rows. This step forms the outermost L.
		\item A column is skipped after the columns that are accommodating the previous l-diagonals, and then l-diagonals are drawn starting from the next column and are continued downwards. A row is skipped after the rows that are accommodating the previous l-diagonals  and then l-diagonals are drawn starting from the row above the skipped row and are continued towards the right.
		$(iii)$ Steps (i) and (ii) are repeated until number of columns and rows that are left is less than l+1.
	\end{enumerate}
\end{definition}
Figure \ref{fig:larr} is an illustration of the L-arrangement of 2-diagonals in an 8$\times$8 array.
%\begin{figure}
%	\centering
%	\includegraphics[width=0.7\linewidth]{Larr}
%	\caption{L-arrangement of 2-diagonals in an 8$\times$8 array}
%	\label{fig:larr}
%\end{figure}

%\begin{figure}[h!]
%	\centering
%	\begin{tikzpicture}
%	\vertex (u1) at (0,0) {d};
%	\vertex (u2) at (4,0)  {c};
%	\vertex (u3) at (4,4)  {b};
%	\vertex (u4) at (0,4) {a};
%	\path 
%	(u1) edge node [below]  {$[0.3, 0.4]$} (u2)
%	(u3) edge node [above]  {$[0.1, 0.2]$} (u4)
%	(u1) edge node [above, rotate=90, pos=0.5]  {$[0.2, 0.3]$} (u4)
%	(u2) edge node [above, rotate=90, pos=0.5]  {$[0.4, 0.5]$} (u3)
%	(u1) edge node [above, rotate=45, pos=0.3]  {$[0.5, 0.6]$} (u3)		
%	;
%	\end{tikzpicture}
%	\caption{Example to show that the converse of proposition \ref{p5} is false}\label{img2}
%\end{figure}

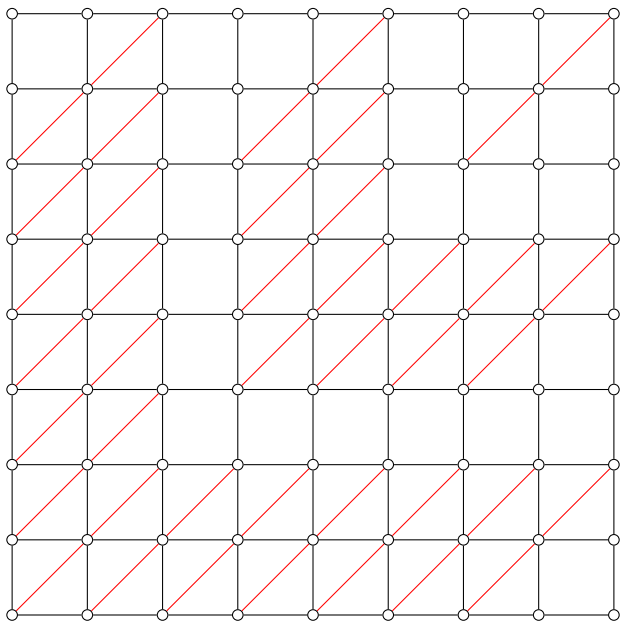
\begin{figure}[h!]
	\centering
	\begin{tikzpicture}
	\vertex (u00) at (0,0) {};
	\vertex (u10) at (1,0) {};
	\vertex (u20) at (2,0) {};
	\vertex (u30) at (3,0) {};	
	\vertex (u40) at (4,0) {};
	\vertex (u50) at (5,0) {};	
	\vertex (u60) at (6,0) {};
	\vertex (u70) at (7,0) {};
	\vertex (u80) at (8,0) {};
	%%%%%	
	\vertex (u01) at (0,1) {};
	\vertex (u11) at (1,1) {};
	\vertex (u21) at (2,1) {};
	\vertex (u31) at (3,1) {};	
	\vertex (u41) at (4,1) {};
	\vertex (u51) at (5,1) {};	
	\vertex (u61) at (6,1) {};
	\vertex (u71) at (7,1) {};
	\vertex (u81) at (8,1) {};
	%%%%%	
	\vertex (u02) at (0,2) {};
	\vertex (u12) at (1,2) {};
	\vertex (u22) at (2,2) {};
	\vertex (u32) at (3,2) {};	
	\vertex (u42) at (4,2) {};
	\vertex (u52) at (5,2) {};	
	\vertex (u62) at (6,2) {};
	\vertex (u72) at (7,2) {};
	\vertex (u82) at (8,2) {};
	%%%%%	
	\vertex (u03) at (0,3) {};
	\vertex (u13) at (1,3) {};
	\vertex (u23) at (2,3) {};
	\vertex (u33) at (3,3) {};	
	\vertex (u43) at (4,3) {};
	\vertex (u53) at (5,3) {};	
	\vertex (u63) at (6,3) {};
	\vertex (u73) at (7,3) {};
	\vertex (u83) at (8,3) {};
	%%%%%
	\vertex (u04) at (0,4) {};
	\vertex (u14) at (1,4) {};
	\vertex (u24) at (2,4) {};
	\vertex (u34) at (3,4) {};	
	\vertex (u44) at (4,4) {};
	\vertex (u54) at (5,4) {};	
	\vertex (u64) at (6,4) {};
	\vertex (u74) at (7,4) {};
	\vertex (u84) at (8,4) {};
	%%%%
	\vertex (u05) at (0,5) {};
	\vertex (u15) at (1,5) {};
	\vertex (u25) at (2,5) {};
	\vertex (u35) at (3,5) {};	
	\vertex (u45) at (4,5) {};
	\vertex (u55) at (5,5) {};	
	\vertex (u65) at (6,5) {};
	\vertex (u75) at (7,5) {};
	\vertex (u85) at (8,5) {};
	%%%%%
	\vertex (u06) at (0,6) {};
	\vertex (u16) at (1,6) {};
	\vertex (u26) at (2,6) {};
	\vertex (u36) at (3,6) {};	
	\vertex (u46) at (4,6) {};
	\vertex (u56) at (5,6) {};	
	\vertex (u66) at (6,6) {};
	\vertex (u76) at (7,6) {};
	\vertex (u86) at (8,6) {};
	%%%%
	\vertex (u07) at (0,7) {};
	\vertex (u17) at (1,7) {};
	\vertex (u27) at (2,7) {};
	\vertex (u37) at (3,7) {};	
	\vertex (u47) at (4,7) {};
	\vertex (u57) at (5,7) {};	
	\vertex (u67) at (6,7) {};
	\vertex (u77) at (7,7) {};
	\vertex (u87) at (8,7) {};
	%%%%%	
	\vertex (u08) at (0,8) {};
	\vertex (u18) at (1,8) {};
	\vertex (u28) at (2,8) {};
	\vertex (u38) at (3,8) {};	
	\vertex (u48) at (4,8) {};
	\vertex (u58) at (5,8) {};	
	\vertex (u68) at (6,8) {};
	\vertex (u78) at (7,8) {};
	\vertex (u88) at (8,8) {};
	
	\path 
	(u00) edge (u01)
	(u01) edge (u02)
	(u02) edge (u03)
	(u03) edge (u04)
	(u04) edge (u05)
	(u05) edge (u06)
	(u06) edge (u07)
	(u07) edge (u08)
	%%%
	
	(u10) edge (u11)
	(u11) edge (u12)
	(u12) edge (u13)
	(u13) edge (u14)
	(u14) edge (u15)
	(u15) edge (u16)
	(u16) edge (u17)
	(u17) edge (u18)
	%%%
	(u20) edge (u21)
	(u21) edge (u22)
	(u22) edge (u23)
	(u23) edge (u24)
	(u24) edge (u25)
	(u25) edge (u26)
	(u26) edge (u27)
	(u27) edge (u28)
	%%%
	(u30) edge (u31)
	(u31) edge (u32)
	(u32) edge (u33)
	(u33) edge (u34)
	(u34) edge (u35)
	(u35) edge (u36)
	(u36) edge (u37)
	(u37) edge (u38)
	%%%
	(u40) edge (u41)
	(u41) edge (u42)
	(u42) edge (u43)
	(u43) edge (u44)
	(u44) edge (u45)
	(u45) edge (u46)
	(u46) edge (u47)
	(u47) edge (u48)
	%%%
	(u50) edge (u51)
	(u51) edge (u52)
	(u52) edge (u53)
	(u53) edge (u54)
	(u54) edge (u55)
	(u55) edge (u56)
	(u56) edge (u57)
	(u57) edge (u58)
	%%%
	
	(u60) edge (u61)
	(u61) edge (u62)
	(u62) edge (u63)
	(u63) edge (u64)
	(u64) edge (u65)
	(u65) edge (u66)
	(u66) edge (u67)
	(u67) edge (u68)
	%%%
	(u70) edge (u71)
	(u71) edge (u72)
	(u72) edge (u73)
	(u73) edge (u74)
	(u74) edge (u75)
	(u75) edge (u76)
	(u76) edge (u77)
	(u77) edge (u78)
	%%%
	(u80) edge (u81)
	(u81) edge (u82)
	(u82) edge (u83)
	(u83) edge (u84)
	(u84) edge (u85)
	(u85) edge (u86)
	(u86) edge (u87)
	(u87) edge (u88)
	%%%
	(u00) edge (u10)
	(u10) edge (u20)
	(u20) edge (u30)
	(u30) edge (u40)
	(u40) edge (u50)
	(u50) edge (u60)
	(u60) edge (u70)
	(u70) edge (u80)
	%%%
	(u01) edge (u11)
	(u11) edge (u21)
	(u21) edge (u31)
	(u31) edge (u41)
	(u41) edge (u51)
	(u51) edge (u61)
	(u61) edge (u71)
	(u71) edge (u81)
	%%%
	(u02) edge (u12)
	(u12) edge (u22)
	(u22) edge (u32)
	(u32) edge (u42)
	(u42) edge (u52)
	(u52) edge (u62)
	(u62) edge (u72)
	(u72) edge (u82)
	%%%
	(u03) edge (u13)
	(u13) edge (u23)
	(u23) edge (u33)
	(u33) edge (u43)
	(u43) edge (u53)
	(u53) edge (u63)
	(u63) edge (u73)
	(u73) edge (u83)
	%%%
	(u04) edge (u14)
	(u14) edge (u24)
	(u24) edge (u34)
	(u34) edge (u44)
	(u44) edge (u54)
	(u54) edge (u64)
	(u64) edge (u74)
	(u74) edge (u84)
	%%%
	(u05) edge (u15)
	(u15) edge (u25)
	(u25) edge (u35)
	(u35) edge (u45)
	(u45) edge (u55)
	(u55) edge (u65)
	(u65) edge (u75)
	(u75) edge (u85)
	%%%
	(u06) edge (u16)
	(u16) edge (u26)
	(u26) edge (u36)
	(u36) edge (u46)
	(u46) edge (u56)
	(u56) edge (u66)
	(u66) edge (u76)
	(u76) edge (u86)
	%%%
	(u07) edge (u17)
	(u17) edge (u27)
	(u27) edge (u37)
	(u37) edge (u47)
	(u47) edge (u57)
	(u57) edge (u67)
	(u67) edge (u77)
	(u77) edge (u87)
	%%%
	(u08) edge (u18)
	(u18) edge (u28)
	(u28) edge (u38)
	(u38) edge (u48)
	(u48) edge (u58)
	(u58) edge (u68)
	(u68) edge (u78)
	(u78) edge (u88)
	%%%
	% Diagonals
	%%
	(u00) edge [color=red] (u11)
	(u10) edge [color=red] (u21)
	(u20) edge [color=red] (u31)
	(u30) edge [color=red] (u41)
	(u40) edge [color=red] (u51)
	(u50) edge [color=red] (u61)
	(u60) edge [color=red] (u71)
	%%%%
	(u01) edge [color=red] (u12)
	(u11) edge [color=red] (u22)
	(u21) edge [color=red] (u32)
	(u31) edge [color=red] (u42)
	(u41) edge [color=red] (u52)
	(u51) edge [color=red] (u62)
	(u61) edge [color=red] (u72)
	(u71) edge [color=red] (u82)
	%%%%
	(u02) edge [color=red] (u13)
	(u03) edge [color=red] (u14)
	(u04) edge [color=red] (u15)
	(u05) edge [color=red] (u16)
	(u06) edge [color=red] (u17)
	%%%%
	(u12) edge [color=red] (u23)
	(u13) edge [color=red] (u24)
	(u14) edge [color=red] (u25)
	(u15) edge [color=red] (u26)
	(u16) edge [color=red] (u27)
	(u17) edge [color=red] (u28)
	%%%%
	(u33) edge [color=red] (u44)
	(u43) edge [color=red] (u54)
	(u53) edge [color=red] (u64)
	(u63) edge [color=red] (u74)
	%%%%
	(u34) edge [color=red] (u45)
	(u44) edge [color=red] (u55)
	(u54) edge [color=red] (u65)
	(u64) edge [color=red] (u75)
	(u74) edge [color=red] (u85)
	%%%%
	(u35) edge [color=red] (u46)
	(u45) edge [color=red] (u56)
	%%%%
	(u36) edge [color=red] (u47)
	(u46) edge [color=red] (u57)
	(u47) edge [color=red] (u58)
	%%%%
	(u66) edge [color=red] (u77)
	(u77) edge [color=red] (u88)
	%%%%
	
	%%%%
	
	%	(u1) edge node [below]  {$[0.3, 0.4]$} (u2)
	%	(u3) edge node [above]  {$[0.1, 0.2]$} (u4)
	%	(u1) edge node [above, rotate=90, pos=0.5]  {$[0.2, 0.3]$} (u4)
	%	(u2) edge node [above, rotate=90, pos=0.5]  {$[0.4, 0.5]$} (u3)
	%	(u1) edge node [above, rotate=45, pos=0.3]  {$[0.5, 0.6]$} (u3)		
	;
	\end{tikzpicture}
	\caption{L-arrangement of 2-diagonals in an 8$\times$8 array}
\end{figure}
\begin{definition} $\cite{b}$
	Let $P_x$ be a path on x vertices. A subpath $P_y$ of $P_x$ is a path on y vertices where $y\le x$
\end{definition}
We use the following notations:\\
$\bullet$ $D_l(n)$ : The maximum number of $l$-diagonals that can be accommodated in an $n$-array.\\
$\bullet$ $L_l(n)$ : The number of $l$-diagonals that can be accommodated in an $n$-array.\\
$\bullet$ $m_l(P_k)$ : The maximum number of subpaths of length $l$ in a path $P_k$, such that every subpath is at a distance of at least one in $P_k$ (Note that the subpath has $l+1$ vertices).

\renewcommand{\thesection}{\large 2}
\section{\large Maximum Number of Non-Intersecting $l$-Diagonals in an $n$-array}

\begin{theorem}\label{th:21}
	For any l-diagonal in an n-array, where n is a multiple of l+1 and l is a positive integer, we have $L_l(n)=\frac{n^2-n(l-2)}{l+1}$
\end{theorem}

\begin{proof}
	Since $n$ is a multiple of $l+1$, we have 
	\begin{equation*}
	n\equiv 0~ mod(l+1) 	\end{equation*}\begin{equation*}\implies n=(l+1)\times a,
	\end{equation*}
	where $a$ is a positive integer.\\
	In accordance with the L-arrangement as defined in Definition \ref{def:Larran}, the first $l-1$ columns and the last $l-1$ rows will not accommodate any $l$-diagonal. Also, a column and a row is to be skipped after every L-formation. Therefore, the number of diagonals in the L-arrangement is:\\
	
	\begin{equation*}
	L_l(n)=(n-(l-1)) + (n-l) + (n-(l+1)-(l-1)) + (n-(l+1)-l)   (n-2(l+1)-(l-1))
	\end{equation*}
	\begin{equation*}
	~~~~~~~~~~~~~~~~~~~~~~+(n-2(l+1)-l) + (n-3(l+1)-(l-1)) + (n-3(l+1)-l) + .....
	\end{equation*}
	\begin{equation*}
	~~~~~~~~~~~~~~~~~~~	..... +(n-(a-1)(l+1)-(l-1)) + (n-(a-1)(l+1)-l)
	\end{equation*}
	
	\begin{equation*}
	\implies L_l(n)=(n-(l-1)) + (n-l) + \sum_{j=1}^{a-1}(n-j(l+1)-(l-1))+ \sum_{j=1}^{a-1}(n-j(l+1)-l)   	
	\end{equation*}
	
	\begin{equation*}
	\implies L_l(n)=2n-2l+1+2n(a-1)-2\bigg(\frac{(a-1)a}{2}\bigg)(l+1)-(a-1)(l-1) -(a-1)l
	\end{equation*}
	
	\begin{equation*}
	\implies L_l(n)=2n-2l+1+(a-1)(2n-a(l+1)-(l-1) -l)
	\end{equation*}\\
	Substituting \begin{math}
	a=\frac{n}{l+1}
	\end{math}, we have,
	
	\begin{equation*}
	~~~~~~~~L_l(n)=2n-2l+1+\big{(\frac{n}{l+1}-1}\big)(2n-n-2l+1)
	\end{equation*}
	
	\begin{equation*}
	\implies L_l(n)=2n-2l+1+\big{(\frac{n-l-1}{l+1}}\big)(n-2l+1)
	\end{equation*}
	
	\begin{equation*}
	\implies L_l(n)=\frac{2n(l+1)-2l(l+1)+(l+1)+(n-l-1)(n-2l+1)}{l+1}
	\end{equation*}
	
	\begin{equation*}
	\implies    L_l(n)=\frac{2n+n^2-nl}{l+1}
	\end{equation*}
	
	\begin{equation*}
	\implies  L_l(n)=\frac{n^2-n(l-2)}{l+1}
	\end{equation*}

\end{proof}
Since this L-arrangement is one form of an L-arrangement, we have $L_l(n)$ as a lower bound for $D_l(n)$,
\begin{equation*}
i.e.,~L_l(n) \le D_l(n)
\end{equation*}
\begin{equation}\label{eq:21}
i.e.,~ D_l(n) \ge \frac{n^2-n(l-2)}{l+1}
\end{equation}

In order to find an upper bound for $D_l(n)$, the concept of path in Graph Theory is being used \cite{b}.

We consider the lattice points of an array to be synonymous with the vertices of a graph as shown in Figure \ref{fig:larr}.
%\begin{figure}[h]
%	\centering
%	\includegraphics[width=0.75\textwidth]{fpath1}
%	\caption{Lattice points of an $n$-array.}
%	\label{fpath1}
%\end{figure}

\begin{figure}[h!]
	\centering
	\begin{tikzpicture}
	\vertex [fill=black] (u00) at (0,0) [label=below:$v_{k,1}$]{};
	\vertex [fill=black] (u10) at (1,0) [label=below:$v_{k,2}$]{};
	%\vertex (u20) at (2,0) {};
	%\vertex (u30) at (3,0) {};	
	%\vertex (u40) at (4,0) {};
	\vertex [fill=black] (u50) at (5,0) [label=below:$v_{k,k-3}$]{};	
	\vertex [fill=black] (u60) at (6,0) [label=below:$v_{k,k-2}$]{};
	\vertex [fill=black] (u70) at (7,0) [label=below:$v_{k,k-1}$]{};
	\vertex [fill=black] (u80) at (8,0) [label=below:$v_{k,k}$]{};
	%%%%%	
	\vertex [fill=black] (u01) at (0,1) [label=below:$v_{k-1,1}$]{};
	\vertex [fill=black] (u11) at (1,1) [label=below:$v_{k-1,2}$]{};
	\vertex [fill=black] (u21) at (2,1) [label=below:$v_{k-1,3}$]{};
	%\vertex (u31) at (3,1) {};	
	%\vertex (u41) at (4,1) {};
	%\vertex (u51) at (5,1) {};	
	\vertex [fill=black] (u61) at (6,1) [label=below:$v_{k-1,k-2}$] {};
	\vertex [fill=black] (u71) at (7,1) [label=below:$v_{k-1,k-1}$] {};
	\vertex [fill=black] (u81) at (8,1) [label=below:$v_{k-1,k}$] {};
	%%%%%	
	\vertex (u02) at (0,2) {};
	\vertex [fill=black] (u12) at (1,2) [label=below:$v_{k-2,2}$] {};
	\vertex [fill=black] (u22) at (2,2) [label=below:$v_{k-2,3}$] {};
	\vertex [fill=black] (u32) at (3,2) [label=below:$v_{k-2,4}$] {};	
	%\vertex (u42) at (4,2) {};
	%\vertex (u52) at (5,2) {};	
	%\vertex (u62) at (6,2) {};
	\vertex [fill=black] (u72) at (7,2) [label=below:$v_{k-2,k-1}$] {};
	\vertex [fill=black] (u82) at (8,2) [label=below:$v_{k-2,k}$] {};
	%%%%%	
	%\vertex (u03) at (0,3) {};
	%\vertex (u13) at (1,3) {};
	\vertex [fill=black] (u23) at (2,3) [label=below:$v_{k-3,3}$] {};
	%\vertex (u33) at (3,3) {};	
	%\vertex (u43) at (4,3) {};
	%\vertex (u53) at (5,3) {};	
	%\vertex (u63) at (6,3) {};
	%\vertex (u73) at (7,3) {};
	\vertex [fill=black] (u83) at (8,3) [label=below:$v_{k-3,k}$] {};
	%%%%%
	%	\vertex (u04) at (0,4) {};
	%	\vertex (u14) at (1,4) {};
	%	\vertex (u24) at (2,4) {};
	%	\vertex (u34) at (3,4) {};	
	%	\vertex (u44) at (4,4) {};
	%	\vertex (u54) at (5,4) {};	
	%	\vertex (u64) at (6,4) {};
	%	\vertex (u74) at (7,4) {};
	%	\vertex (u84) at (8,4) {};
	%%%%
	\vertex [fill=black] (u05) at (0,5) [label=below:$v_{4,1}$] {};
	%	\vertex (u15) at (1,5) {};
	%	\vertex (u25) at (2,5) {};
	%	\vertex (u35) at (3,5) {};	
	%	\vertex (u45) at (4,5) {};
	%	\vertex (u55) at (5,5) {};	
	%	\vertex [fill=black] (u65) at (6,5) [label=below:$v_{4,k-2}$] {};
	%	\vertex (u75) at (7,5) {};
	%	\vertex (u85) at (8,5) {};
	%%%%%
	\vertex [fill=black] (u06) at (0,6) [label=below:$v_{3,1}$] {};
	\vertex [fill=black] (u16) at (1,6) [label=below:$v_{3,2}$] {};
	%	\vertex (u26) at (2,6) {};
	%	\vertex (u36) at (3,6) {};	
	%	\vertex (u46) at (4,6) {};
	\vertex [fill=black] (u56) at (5,6) [label=below:$v_{3,k-3}$] {};	
	\vertex [fill=black] (u66) at (6,6) [label=below:$v_{3,k-2}$] {};
	\vertex [fill=black] (u76) at (7,6) [label=below:$v_{3,k-1}$] {};
	%\vertex (u86) at (8,6) {};
	%%%%
	\vertex [fill=black] (u07) at (0,7) [label=below:$v_{2,1}$] {};
	\vertex [fill=black] (u17) at (1,7) [label=below:$v_{2,2}$] {};
	\vertex [fill=black] (u27) at (2,7) [label=below:$v_{2,3}$] {};
	%	\vertex (u37) at (3,7) {};	
	%	\vertex (u47) at (4,7) {};
	%	\vertex (u57) at (5,7) {};	
	\vertex [fill=black] (u67) at (6,7) [label=below:$v_{2,k-2}$] {};
	\vertex [fill=black] (u77) at (7,7) [label=below:$v_{2,k-1}$] {};
	\vertex [fill=black] (u87) at (8,7) [label=below:$v_{2,k}$] {};
	%%%%%	
	\vertex [fill=black] (u08) at (0,8) [label=below:$v_{1,1}$] {};
	\vertex [fill=black] (u18) at (1,8) [label=below:$v_{1,2}$] {};
	\vertex [fill=black] (u28) at (2,8) [label=below:$v_{1,3}$] {};
	\vertex [fill=black] (u38) at (3,8) [label=below:$v_{1,4}$] {};	
	%	\vertex (u48) at (4,8) {};
	%	\vertex (u58) at (5,8) {};	
	%	\vertex (u68) at (6,8) {};
	\vertex [fill=black] (u78) at (7,8) [label=below:$v_{1,k-1}$] {};
	\vertex [fill=black] (u88) at (8,8) [label=below:$v_{1,k}$] {};

	\vertex [fill,circle,inner sep=0pt,minimum size=1pt](u29) at (2,9) [label=$P_{k-2}$] {};
	\vertex [fill,circle,inner sep=0pt,minimum size=1pt](u290) at (1.3,8.3) {};
	\vertex [fill,circle,inner sep=0pt,minimum size=1pt](u39) at (3,9) [label=$P_{k-3}$] {};
	\vertex [fill,circle,inner sep=0pt,minimum size=1pt](u390) at (2.3,8.3) {};
	\vertex [fill,circle,inner sep=0pt,minimum size=1pt](u49) at (4,9) [label=$P_{k-4}$] {};
	\vertex [fill,circle,inner sep=0pt,minimum size=1pt](u490) at (3.3,8.3) {};
	\vertex [fill,circle,inner sep=0pt,minimum size=1pt](u89) at (8,9) [label=$P_{1}$]  {};
	\vertex [fill,circle,inner sep=0pt,minimum size=1pt](u890) at (7.3,8.3) {};
	%\vertex [fill=black](u89) at (8,9) [label=$P_{1}$] {};
	\vertex [fill,circle,inner sep=0pt,minimum size=1pt](u99) at (9,9) [label=$P_{0}$] {};
	\vertex [fill,circle,inner sep=0pt,minimum size=1pt](u990) at (8.3,8.3) {};
	\vertex [fill,circle,inner sep=0pt,minimum size=1pt](u98) at (9,8) [label=$P'_{1}$] {};
	\vertex [fill,circle,inner sep=0pt,minimum size=1pt](u980) at (8.3,7.3) {};
	\vertex [fill,circle,inner sep=0pt,minimum size=1pt](u92) at (9,2) [label=$P'_{k-2}$] {};
	\vertex [fill,circle,inner sep=0pt,minimum size=1pt](u920) at (8.3,1.3) {};
	\vertex [fill,circle,inner sep=0pt,minimum size=1pt](u93) at (9,3) [label=$P'_{k-3}$] {};
	\vertex [fill,circle,inner sep=0pt,minimum size=1pt](u930) at (8.3,2.3) {};
	\vertex [fill,circle,inner sep=0pt,minimum size=1pt](u94) at (9,4) [label=$P'_{k-4}$] {};
	\vertex [fill,circle,inner sep=0pt,minimum size=1pt](u940) at (8.3,3.3) {};
	\path 
	(u07) edge (u18)
	(u06) edge  (u28)
	(u05) edge  (u38)	
	(u04) edge [dashed] (u48)
	(u03) edge [dashed] (u58)
	(u02) edge [dashed] (u68)
	%(u02) edge [dashed] (u78)
	(u01) edge (u78)
	(u00) edge (u88)
	(u10) edge (u87)
	(u20) edge [dashed] (u86)
	(u30) edge [dashed] (u85)
	(u40) edge [dashed] (u84)
	(u50) edge  (u83)	
	(u60) edge  (u82)	
	(u70) edge  (u81)		
	;
	\path[->] (u29) edge [dashed] (u290);
	\path[->] (u39) edge [dashed] (u390);
	\path[->] (u49) edge [dashed] (u490);
	\path[->] (u89) edge [dashed] (u890);
	\path[->] (u99) edge [dashed] (u990);
	\path[->] (u98) edge [dashed] (u980);
	\path[->] (u92) edge [dashed] (u920);
	\path[->] (u93) edge [dashed] (u930);
	\path[->] (u94) edge [dashed] (u940);
	\end{tikzpicture}
	\caption{Lattice points of an $n$-array}
\end{figure}
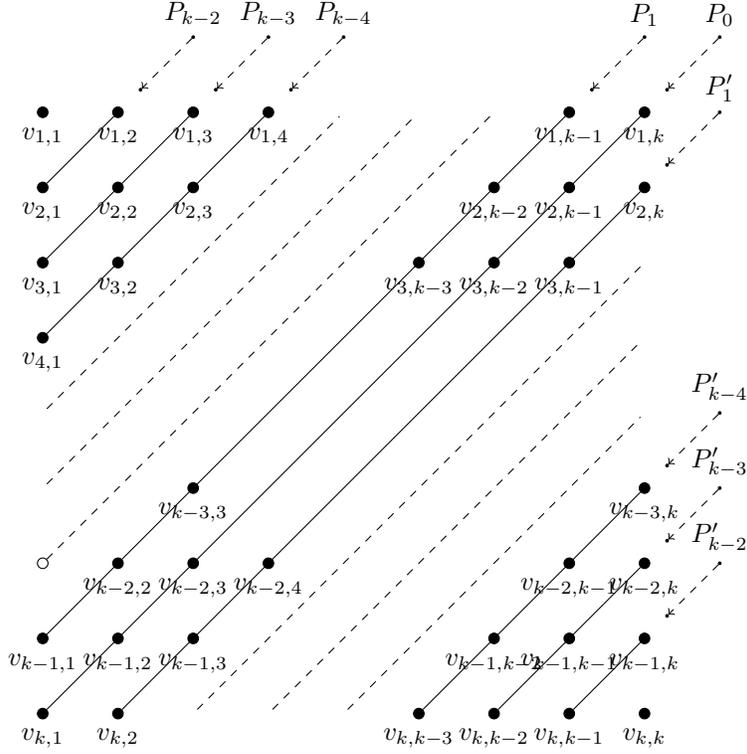

Let $v_{i,j}$ denote the lattice points of the array. Since we are dealing with positive diagonals only, then any $l$-diagonal can be understood as a path having $l$ edges and $l$+1 vertices. As shown in Figure 1 , $P_{r}$ and $P_{r}^{\prime}$ are paths of length $k-r$, where $2\le r \le k$ and $P_{r} = P_{r}^{\prime}$ when $r = 0$. In order to obtain an $l$-diagonal in the path $P_{r}$, we need to find a sub-path $P_{l+1}$ of $P_{r}$. Similarly with path $P_{r}^{\prime}$.\\
From the figure, we observe the following:\\
$\bullet$ $P_r = \{v_{p,k-q-p+1}\}$, where $1\le p\le k$, $1\le q\le k-2$ and $k-q-p+1 > 0$.\\
$\bullet$ $P_{r}^{\prime} = \{v_{r+1,k-p+1}\}$, where 1$\le p\le k-1$, $1\le r\le k-2$ and $k-p+1>0$ .  \\
Thus, as can be seen in Figure \ref{fpath1}, an $n$-array can be represented by a $k \times k$ matrix, where $k = n+1$.\\
\newpage
We shall prove a lemma that will assist us in obtaining our main aim.
\begin{lemma}\label{lem:subpath}
	For every positive integer k $\ge$ 2, let P$_k$ be a path with k vertices. Then, the maximum number of subpaths of length $l$ in $P_k$, i.e., $m_l(P_k)$ = $ \big{\lfloor} \frac{k}{l+1} \big{\rfloor}$.
\end{lemma}
\begin{proof}
	First, we find the number of subpaths $P^{\prime}$ of length $l$ in $P_k$, where every $P^{\prime}$ is at a distance one. Let $q$ denote this number.
	\\
	Case 1.
	When $l = k-1$ ($i.e., $ length of the subpath is equal to length of the path).\\
	Since the total number of vertices in $P_l$ is $l$, the length of $P_l$ is $l-1$ = $k$.\\
	Clearly, we have exactly one subpath $P_l$ in $P_k$, which gives
	\begin{equation}\label{eq:22}
	q = 1
	\end{equation}
	Now,
	\begin{equation}\label{eq:23}
	\bigg{\lfloor} \frac{k}{l+1} \bigg{\rfloor} = 	\bigg{\lfloor} \frac{k}{k-1+1} \bigg{\rfloor} = 	 \lfloor 1 \rfloor = 1 
	\end{equation}
	From equations (\ref{eq:22}) and (\ref{eq:23}), we have 
	\begin{equation*}
	q = \bigg{\lfloor} \frac{k}{l+1} \bigg{\rfloor}
	\end{equation*}
	Case 2.
	When $l < k-1$ ($i.e., $ length of the subpath is less than length of the path).\\
	Every subpath $P_l$ involves $l+1$ vertices.\\
	Therefore, number of subpaths $P_l$ is equal to the number of vertices in $P_k$ divided by the number of vertices in each $P_k$, $i.e.,$\\
	\begin{math}
	q = \frac{k}{l+1}
	\end{math}
	\flushleft
	
	By Euclidean Algorithm, 
	
	\begin{align*}
	k = (l + 1)q + r,~ where~ 0 \le r < l + 1.
	\end{align*}
	Case 2.1.
	If $r = 0$, then
	\begin{align*}
	k = (l+1)q
	\implies \frac{k}{l + 1} = q.
	\end{align*}

	\flushleft	Case 2.2.
	If $r \ne 0$, then
	\begin{align*}
	k = (l+1)q + r.\\
	\implies \frac{k}{l + 1} = q + \frac{r}{l + 1}.\\
	\end{align*}
	Since $r < l + 1$, we have, $$ 
	\frac{r}{l + 1} < 1
	$$
	
	Therefore from Case 1 and Case 2, the number of subpaths $P^{\prime}$ where every subpath is at a distance one is,
	
	\begin{align*}
	q =  \bigg{\lfloor} \frac{k}{l+1} \bigg{\rfloor}
	\end{align*}

	Now we claim that the maximum number of subpaths $P^{\prime}$ occur when every subpath is at a distance 1, and this will conclude our proof.\\
	On the contrary, suppose that the maximum number of subpaths occur in such a way that some subpaths are not at a distance one.\\
	We first consider the case where $k=(l+1)q$.\\
	
	In this case, if two subpaths $P_k$ are not at a distance one, then the number of subpaths $P_k$ wil be at least 1 less than q, $i.e.$, at least \begin{math}
	\big\lfloor \frac{k}{l+1}\big\rfloor - 1
	\end{math}
	\\This is a contradiction since number of subpaths where every edge is at a distance one is 
	\begin{align*}
	\bigg\lfloor \frac{k}{l+1}\bigg\rfloor > \bigg\lfloor \frac{k}{l+1}\bigg\rfloor - 1
	\end{align*}
	Next, we consider the case wher $k = (l+1)q + r$, where $0 < r < l+1$.\\
	If the number of vertices which are not in any of the subpaths $\overline{P}$ is greater than $r$, then the number of subpaths will be at least 1 less than $k$, $i.e.,$ \begin{math}
	\lfloor\frac{k}{l+1}\rfloor - 1
	\end{math}\\
	This is again a contradiction as in the above case.\\
	If the number of vertices which are not in any of the subpaths $P^{\prime}$ is less than or equal to $r$, then the number of subpaths will be \begin{math}
	\lfloor\frac{k}{l+1}\rfloor
	\end{math}\\
	This is the same as $q$, where every subpath $P^{\prime}$ is at a distance one.\\
	Therefore the claim holds, and hence the result.\\
\end{proof}

Applying Lemma \ref{lem:subpath} on each of the paths $P_{r}$ and $P_{r}^{\prime}$ for every $r=2,3,4,...,k$, we will get the maximum number of non-intersecting $l$-diagonals for each of the paths. Any arrangement of the $l$-diagonals in each path $P_{r}$ and $P_{r}^{\prime}$ will never exceed $m_l(P_r)$ and $m_l(P_{r}^{\prime})$ respectively for all $r$.\\
From triangular inequality of real numbers, for any two real number $x$ and $y$, we have 
\begin{equation*}
|x|+|y| \le |x+y|
\end{equation*}
But if $x$ and $y$ are non-negative, we have,
\begin{equation}\label{eq:24}
|x|+|y| = |x+y|
\end{equation}
Since the number of diagonals is a non-negative number then by equation (\ref{eq:24}), the sum of the number of diagonals in each path is equal to the total number of diagonals in all the paths. Since number of $l$-diagonals in each path $P_{r}$ and $P_{r}^{\prime}$ will never exceed $m_l(P_r)$  and $m_l(P_{r}^{\prime})$ respectively, for all $r$, therefore, the number of $l$-diagonals in all the paths will not exceed the sum of  $m_l(P_r)$ and $m_l(P_{r}^{\prime})$ for all $r$. This will give an upper bound for the maximum number of non-intersecting $l$-diagonals in an $n$-array.\\
\begin{lemma}\label{lem:D1}
	$D_1(n) = \frac{n^2+n}{2}$, where $n \equiv~0~(mod~2)$.
\end{lemma} 	
\begin{proof}
	By a result \cite{a}, for any $n$-square matrix where $n$ is even,  we have, $D_1(n) = L_1(n)$.\\
	Therefore, using Theorem \ref{th:21}, we have,
	
	\begin{equation*}
	~~~~~~~~ D_1(n) = \frac{n^2-n(1-2)}{1+1}
	\end{equation*}
	\begin{equation*}
	\implies D_1(n) = \frac{n^2+n}{2}
	\end{equation*}
\end{proof}
\begin{lemma}\label{lem:D2}
	$D_2(n) \le \frac{n^2}{3}$, where n is a positive integer and $n \equiv~0~(mod~3)$.
\end{lemma}
\begin{proof}
	Let $M$ denote the number of non-intersecting 2-diagonals in any arrangement in an $n$ array. This array can be represented as a $k \times k$ matrix, where $k=n+1$.\\
	Now, $n \equiv0(mod~ 3) \implies n = 3a$, where $a$ is a positive integer.
	By Lemma 1, we have,
	
	\begin{equation}\label{eq:25}
	M \le \bigg\lfloor\frac{k}{3}\bigg\rfloor+2\bigg\lfloor\frac{k-1}{3}\bigg\rfloor+2\bigg\lfloor\frac{k-2}{3}\bigg\rfloor+2\bigg\lfloor\frac{k-3}{3}\bigg\rfloor+...
	...+2\bigg\lfloor\frac{4}{3}\bigg\rfloor+2\bigg\lfloor\frac{3}{3}\bigg\rfloor+2\bigg\lfloor\frac{2}{3}\bigg\rfloor
	\end{equation}
	Each term inside the floor function on the right hand side of (2.5) can be represented as\\
	\begin{math}
	\big\lfloor\frac{k-x}{3}\big\rfloor,
	\end{math} where $x=1,2,3,....,k-2.$\\
	We have the following cases:\\
	Case 1:
	When \begin{math}
	x \equiv 0 (mod~3)
	\implies x=3z,
	\end{math}
	where $z$ is a positive integer.\\
	Then,\\
	\begin{equation*}
	\bigg\lfloor\frac{k-x}{3}\bigg\rfloor = 	 \bigg\lfloor\frac{3a +1-3z}{3}\bigg\rfloor =  	 \bigg\lfloor(a-z)+\frac{1}{3}\bigg\rfloor = a-z
	\end{equation*}
	Case 2: 
	When \begin{math}
	x \equiv 1 (mod~3)
	\implies x=3z+1,
	\end{math}
	where $z$ is a positive integer.\\
	Then,\\
	\begin{equation*}
	\bigg\lfloor\frac{k-x}{3}\bigg\rfloor = 	 \bigg\lfloor\frac{3a +1-3z-1}{3}\bigg\rfloor =  	 \big\lfloor(a-z)\big\rfloor = a-z
	\end{equation*}
	Case 3: 
	When \begin{math}
	x \equiv 2 (mod~3)
	\implies x=3z+2,
	\end{math}
	where $z$ is a positive integer.\\
	Then,\\
	\begin{equation*}
	\bigg\lfloor\frac{k-x}{3}\bigg\rfloor = 	 \bigg\lfloor\frac{3a +1-3z-2}{3}\bigg\rfloor =  	 \bigg\lfloor(a-z)-\frac{1}{3}\bigg\rfloor = a-z-1
	\end{equation*}

	Therefore, we have the following:\\
	When $z=0,$\\ \begin{math}
	By~ Case~ 1,~ 3z =0 \implies x=0\implies \big\lfloor\frac{k-x}{3}\big\rfloor = \big\lfloor\frac{k}{3}\big\rfloor = a-0\\
	By~ Case~ 2,~3z+1=1 \implies x=1 \implies \big\lfloor\frac{k-x}{3}\big\rfloor= \big\lfloor\frac{k-1}{3}\big\rfloor =  a-0 \\
	By~ Case~ 3,~3z+2=2 \implies x=2 \implies \big\lfloor\frac{k-x}{3}\big\rfloor= \big\lfloor\frac{k-2}{3}\big\rfloor= a-0-1 
	\end{math}\\
	When $z=1,$\\ \begin{math}
	By~ Case~ 1,~ 3z =3 \implies x=3\implies \big\lfloor\frac{k-x}{3}\big\rfloor = \big\lfloor\frac{k-3}{3}\big\rfloor =  a-1\\
	By~ Case~ 2,~ 3z+1=4 \implies x=4 \implies \big\lfloor\frac{k-x}{3}\big\rfloor= \big\lfloor\frac{k-4}{3}\big\rfloor=  a-1 \\
	By~ Case~ 3,~ 3z+2=5 \implies x=5 \implies \big\lfloor\frac{k-x}{3}\big\rfloor= \big\lfloor\frac{k-5}{3}\big\rfloor= a-1-1 
	\end{math}\\
	.\\
	.\\
	.\\
	When $z=\frac{k}{3}-2,$\\ \begin{math}
	By~ Case~ 1,~ 3z =k-6 \implies x=k-6\implies \big\lfloor\frac{k-x}{3}\big\rfloor = \big\lfloor\frac{6}{3}\big\rfloor =  a-\big(\frac{k}{3}-2\big)-1\\
	By~ Case~ 2,~ 3z +1 =k-5 \implies x=k-5\implies \big\lfloor\frac{k-x}{3}\big\rfloor = \big\lfloor\frac{5}{3}\big\rfloor =  a-\big(\frac{k}{3}-2\big)\\
	By~ Case~ 3,~ 3z +2 =k-4 \implies x=k-4\implies \big\lfloor\frac{k-x}{3}\big\rfloor = \big\lfloor\frac{4}{3}\big\rfloor =  a-\big(\frac{k}{3}-2\big)\\
	\end{math}
	When $z=\frac{k}{3}-1,$\\ \begin{math}
	By~ Case~ 1,~ 3z =k-3 \implies x=k-3\implies \big\lfloor\frac{k-x}{3}\big\rfloor = \big\lfloor\frac{3}{3}\big\rfloor =  a-\big(\frac{k}{3}-1\big)-1\\
	By~ Case~ 2,~ 3z +1 =k-2 \implies x=k-2\implies \big\lfloor\frac{k-x}{3}\big\rfloor = \big\lfloor\frac{2}{3}\big\rfloor =  a-\big(\frac{k}{3}-1\big)\\
	\end{math}\\
	Substituting in equation (\ref{eq:25}), we get,
	\begin{equation*}
	M \le (a-0) + 2(a-0)+2(a-0-1) + 2(a-1)+2(a-1)+2(a-1-1) +....
	\end{equation*}
	\begin{equation*}
	+ 2\bigg(a-\bigg(\frac{k}{3}-2\bigg)\bigg)+ 2\bigg(a-\bigg(\frac{k}{3}-2\bigg)\bigg)+ 2\bigg(a-\bigg(\frac{k}{3}-2\bigg)-1\bigg)
	+ 2\bigg(a-\bigg(\frac{k}{3}-1\bigg)\bigg)+ 2\bigg(a-\bigg(\frac{k}{3}-1\bigg)\bigg)
	\end{equation*} 
	\begin{equation*}
	=9a+2-\frac{4}{3}k + 4\sum_{t=1}^{\frac{k}{3}-2}4(a-t) + \sum_{t=1}^{\frac{k}{3}-2}2(a-t-1)
	\end{equation*}
	\begin{equation*}
	=9a+2-\frac{4}{3}k +2(a-1)\bigg(\frac{k}{3}-2\bigg)-\frac{2}{2}\bigg(\frac{k}{3}-2\bigg)\bigg(\frac{k}{3}-1\bigg)
	+4a\bigg(\frac{k}{3}-2\bigg)-\frac{4}{2}\bigg(\frac{k}{3}-2\bigg)\bigg(\frac{k}{3}-1\bigg)
	\end{equation*}
	\begin{equation*}
	=9a+2-\frac{4}{3}k+\bigg(\frac{k}{3}-2\bigg)\bigg(2(a-1)-\bigg(\frac{k}{3}-1\bigg)+4a-\bigg(\frac{k}{3}-1\bigg)\bigg)
	\end{equation*}
	\begin{equation*}
	=\frac{1}{3}[27a +6 -4k+(k-6)(6a-k+1) ]
	\end{equation*}
	Substituting the values of $a$ and $k$ in terms of $n$, we get,
	\begin{equation*}
	=\frac{1}{3}[5n+2+(n-5)n]
	\end{equation*}
	\begin{equation*}
	=\frac{1}{3}(n^2+2)
	\end{equation*}
	Thus,
	\begin{equation}\label{eq:26}
	M\le \frac{n^2+2}{3}
	\end{equation}
	However, the number of non-intersecting $l$-diagonals is always a non-negative integer. Therefore, we have put  equation (\ref{eq:26}) as
	\begin{equation*}
	M\le \bigg\lfloor \frac{n^2+2}{3}\bigg\rfloor = \bigg\lfloor\frac{n^2}{3}+\frac{2}{3}\bigg\rfloor = \frac{n^2}{3}
	\end{equation*}
	since 3 is a factor of $n$.
	\\Therefore, 
	
	\begin{equation*}
	D_2(n) \le \frac{n^2}{3}
	\end{equation*}

\end{proof}

\begin{lemma}\label{lem:D3}
	$D_3(n) \le \frac{n^2-n}{4}$, where n is a positive integer and $n \equiv0~(mod~4)$.
\end{lemma}

\begin{proof}
	Let $M$ denote the number of non-intersecting 3-diagonals in any arrangement in an $n$ array. This array can be represented as a $k \times k$ matrix, where $k=n+1$.\\
	Now, $n \equiv0(mod~ 4) \implies n = 4a$, where $a$ is a positive integer. \\ Using Lemma 1, we have,
	
	\begin{equation}\label{eq:27}
	M \le \bigg\lfloor\frac{k}{4}\bigg\rfloor+
	2\bigg\lfloor\frac{k-1}{4}\bigg\rfloor+2\bigg\lfloor\frac{k-2}{4}
	\bigg\rfloor+2\bigg\lfloor\frac{k-3}{4}\bigg\rfloor+...
	...+2\bigg\lfloor\frac{4}{4}\bigg\rfloor+2\bigg\lfloor\frac{3}{4}
	\bigg\rfloor+2\bigg\lfloor\frac{2}{4}\bigg\rfloor
	\end{equation}
	Each term inside the floor function on the right hand side of equation (\ref{eq:27}) can be represented as\\
	\begin{math}
	\big\lfloor\frac{k-x}{4}\big\rfloor,
	\end{math} where $x=1,2,3,....,k-2.$\\
	We have the following cases:\\
	Case 1:
	When \begin{math}
	x \equiv 0 (mod~4)
	\implies x=4z,
	\end{math}
	where $z$ is a positive integer.\\
	Then,\\
	\begin{equation*}
	\bigg\lfloor\frac{k-x}{4}\bigg\rfloor = 	 \bigg\lfloor\frac{4a +1-4z}{4}\bigg\rfloor = \bigg\lfloor(a-z)+\frac{1}{4}\bigg\rfloor = a-z
	\end{equation*}
	Case 2: 
	When \begin{math}
	x \equiv 1 (mod~4)
	\implies x=4z+1,
	\end{math}
	where $z$ is a positive integer.\\
	Then,\\
	\begin{equation*}
	\bigg\lfloor\frac{k-x}{3}\bigg\rfloor = 	 \bigg\lfloor\frac{4a +1-4z-1}{4}\bigg\rfloor =  	 \big\lfloor(a-z)\big\rfloor = a-z
	\end{equation*}
	Case 3: 
	When \begin{math}
	x \equiv 2 (mod~4)
	\implies x=4z+2,
	\end{math}
	where $z$ is a positive integer.\\
	Then,\\
	\begin{equation*}
	\bigg\lfloor\frac{k-x}{4}\bigg\rfloor = 	 \bigg\lfloor\frac{4a +1-4z-2}{4}\bigg\rfloor =  	 \bigg\lfloor(a-z)-\frac{1}{4}\bigg\rfloor = a-z-1
	\end{equation*}
	Case 4: 
	When \begin{math}
	x \equiv 3 (mod~4)
	\implies x=4z+3,
	\end{math}
	where $z$ is a positive integer.\\
	Then,\\
	\begin{equation*}
	\bigg\lfloor\frac{k-x}{4}\bigg\rfloor = 	 \bigg\lfloor\frac{4a +1-4z-3}{4}\bigg\rfloor =  	 \bigg\lfloor(a-z)-\frac{1}{4}\bigg\rfloor = a-z-1
	\end{equation*}
	
	Therefore, we have the following:\\
	When $z=0,$\\ \begin{math}
	By~ Case~ 1,~ 4z =0 \implies x=0\implies \big\lfloor\frac{k-x}{4}\big\rfloor = \big\lfloor\frac{k}{3}\big\rfloor = a-0\\
	By~ Case~ 2,~4z+1=1 \implies x=1 \implies \big\lfloor\frac{k-x}{4}\big\rfloor= \big\lfloor\frac{k-1}{4}\big\rfloor =  a-0 \\
	By~ Case~ 3,~4z+2=2 \implies x=2 \implies \big\lfloor\frac{k-x}{4}\big\rfloor= \big\lfloor\frac{k-2}{4}\big\rfloor= a-0-1 \\
	By~ Case~ 4,~4z+3=3 \implies x=3 \implies \big\lfloor\frac{k-x}{4}\big\rfloor= \big\lfloor\frac{k-2}{4}\big\rfloor= a-0-1 		
	\end{math}\\
	When $z=1,$\\ \begin{math}
	By~ Case~ 1,~ 4z =4 \implies x=4\implies \big\lfloor\frac{k-x}{4}\big\rfloor = \big\lfloor\frac{k-4}{3}\big\rfloor = a-1\\
	By~ Case~ 2,~4z+1=5 \implies x=5 \implies \big\lfloor\frac{k-x}{4}\big\rfloor= \big\lfloor\frac{k-5}{4}\big\rfloor =  a-1 \\
	By~ Case~ 3,~4z+2=6 \implies x=2 \implies \big\lfloor\frac{k-x}{4}\big\rfloor= \big\lfloor\frac{k-6}{4}\big\rfloor= a-1-1 \\
	By~ Case~ 4,~4z+3=7 \implies x=3 \implies \big\lfloor\frac{k-x}{4}\big\rfloor= \big\lfloor\frac{k-7}{4}\big\rfloor= a-1-1 	
	\end{math}\\
	.\\
	.\\
	.\\
	When $z=\frac{k}{4}-2,$\\ \begin{math}
	By~ Case~ 1,~ 4z =k-8 \implies x=4\implies \big\lfloor\frac{k-x}{4}\big\rfloor = \big\lfloor\frac{8}{4}\big\rfloor = a-\big(\frac{k}{4}-2 \big)\\
	By~ Case~ 2,~4z +1=k-7 \implies x=4\implies \big\lfloor\frac{k-x}{4}\big\rfloor = \big\lfloor\frac{7}{4}\big\rfloor = a-\big(\frac{k}{4}-2 \big)\\
	By~ Case~ 3,~4z +2=k-6 \implies x=4\implies \big\lfloor\frac{k-x}{4}\big\rfloor = \big\lfloor\frac{6}{4}\big\rfloor = a-\big(\frac{k}{4}-2 \big)-1\\
	By~ Case~ 4,~4z +2=k-5 \implies x=4\implies \big\lfloor\frac{k-x}{4}\big\rfloor = \big\lfloor\frac{5}{4}\big\rfloor = a-\big(\frac{k}{4}-2 \big)-1
	\end{math}\\
	When $z=\frac{k}{4}-1,$\\
	\begin{math}
	By~ Case~ 1,~ 4z =k-4 \implies x=4\implies \big\lfloor\frac{k-x}{4}\big\rfloor = \big\lfloor\frac{4}{4}\big\rfloor = a-\big(\frac{k}{4}-2 \big)\\
	By~ Case~ 2,~4z +1=k-3 \implies x=4\implies \big\lfloor\frac{k-x}{4}\big\rfloor = \big\lfloor\frac{3}{4}\big\rfloor = a-\big(\frac{k}{4}-2 \big)\\
	By~ Case~ 3,~4z +2=k-2 \implies x=4\implies \big\lfloor\frac{k-x}{4}\big\rfloor = \big\lfloor\frac{2}{4}\big\rfloor = a-\big(\frac{k}{4}-2 \big)-1\\
	\end{math}\\
	Substituting in equation (\ref{eq:27}), we get,
	\begin{equation*}
	M \le (a-0) + 2(a-0)+2(a-0-1)+2(a-0-1) + 2(a-1)+2(a-1)+2(a-1-1)
	\end{equation*}
	\begin{equation*}
	~~~~	+2(a-1-1)+......+ 2\bigg(a-\bigg(\frac{k}{4}-2\bigg)\bigg)+ 2\bigg(a-\bigg(\frac{k}{4}-2\bigg)\bigg)+ 2\bigg(a-\bigg(\frac{k}{4}-2\bigg)-1\bigg)
	\end{equation*} 
	\begin{equation*}
	~~~~	+2\bigg(a-\bigg(\frac{k}{4}-2\bigg)-1\bigg)
	+ 2\bigg(a-\bigg(\frac{k}{4}-1\bigg)\bigg)+ 2\bigg(a-\bigg(\frac{k}{4}-1\bigg)\bigg)+
	2\bigg(a-\bigg(\frac{k}{4}-1\bigg)-1\bigg)
	\end{equation*}
	\begin{equation*}
	=13a-\frac{3}{2}k + \sum_{t=1}^{\frac{k}{3}-2}4(a-t) + 4\sum_{t=1}^{\frac{k}{3}-2}4(a-t-1)
	\end{equation*}
	\begin{equation*}
	=13a-\frac{3}{2}k 
	+4a\bigg(\frac{k}{4}-2\bigg)-\frac{4}{2}\bigg(\frac{k}{4}-2\bigg)\bigg(
	\frac{k}{4}-1\bigg)+4(a-1)\bigg(\frac{k}{4}-2\bigg)
	\end{equation*}
	\begin{equation*}
	~~~-\frac{4}{2}\bigg(\frac{k}{4}-2\bigg)\bigg(\frac{k}{4}-1\bigg)
	\end{equation*}
	\begin{equation*}
	=13a-\frac{3}{2}k+\bigg(\frac{k}{4}-2\bigg)\bigg(4a-2\bigg(\frac{k}{4}-
	1\bigg)+4(a-1)-\bigg(\frac{k}{4}-1\bigg)\bigg)
	\end{equation*}
	\begin{equation*}
	=\frac{1}{4}[52a -6k+(k-8)(8a-k) ]
	\end{equation*}
	Substituting the values of $a$ and $k$ in terms of $n$, we get,
	\begin{equation*}
	=\frac{1}{4}[7n-6+(n-7)(n-1)]
	\end{equation*}
	\begin{equation*}
	=\frac{1}{4}(n^2-n+1)
	\end{equation*}
	Thus,
	
	\begin{equation}\label{eq:28}
	M\le \frac{1}{4}(n^2-n+1)
	\end{equation}
	However, the number of non-intersecting $l$-diagonals is always a non-negative integer. Therefore, we have equation (\ref{eq:28}) as
	\begin{equation*}
	M\le \bigg\lfloor \frac{1}{4}(n^2-n+1)\bigg\rfloor = \bigg\lfloor\frac{n^2-n}{4}+\frac{1}{4}\bigg\rfloor = \frac{n^2-n}{4},
	\end{equation*}
	since 4 is a factor of $n$.
	
	Therefore, 
	
	\begin{equation*}
	D_3(n) \le \frac{n^2-n}{4}
	\end{equation*}
	
\end{proof}

From Lemma \ref{lem:D1}, Lemma \ref{lem:D2} and Lemma \ref{lem:D3}, we observe that the results follow a pattern and hence we can get a general result,
\begin{equation}\label{eq:29}
D_l(n)\le \frac{n^2-n(l-2)}{l+1}
\end{equation}
From equations (\ref{eq:21}) and (\ref{eq:29}), we have 
\begin{equation*}
D_l(n) = \frac{n^2-n(l-2)}{l+1}
\end{equation*}

%The content of the section.\\
%It is preferable if the following definitions are used to write theorems, corollaries, lemmas, etc.:
%\begin{lemma}
%A formulation of the lemma
%\end{lemma}
%Proof\\
%\begin{theorem}
%A formulation of the theorem
%\end{theorem}
%Proof\\
%
%If you do not want to enumerate theorems, lemmas, propositions, etc. then the following definition would be used:
%\begin{Lemma}
%A formulation of the lemma
%\end{Lemma}
%Proof\\
%\begin{Theorem}
%A formulation of the theorem
%\end{Theorem}
%Proof\\
%
%\renewcommand{\thesection}{\large 3}
%\section{\large The title of the third section}
%
%The content of the section

\renewcommand{\thesection}{\large 3}
\section{\large Conclusion}

In this paper, we have derived a formula for the maximum number of non-intersecting diagonals in an $n \times n$ array, for any arbitrary length $l$ of the diagonal and for $n \equiv~0~(mod~(l+1))$.

In the process of deriving this formula, a result for the maximum number of independent subpaths in a given path has also been obtained. This result was used to assist in attaining the desired formula.

%%%%%%%%%%%%%%%%%%%%%%%%%%%%%%%%%%%%%%%%%%%%%%%%%%%%%%%%%%%%%%%%%%%%%%%%%%%
%\thispagestyle{headings}

\end{document}